\documentclass[preprint,10pt]{elsarticle}

\usepackage{amsmath}
\usepackage{amssymb}
\usepackage{amsthm}

\newtheorem{thm}{Theorem}[section]

\newtheorem{rem}[thm]{Remark}
\numberwithin{equation}{section}

\begin{document}

\begin{frontmatter}



\title{Several inequalities concerning interpolation in classical Fourier analysis}


\author{Runzhe Zhang and Hua Wang \footnote{E-mail address: 3089486628@qq.com, wanghua@pku.edu.cn.}\\
\textbf{Dedicated to the memory of Li Xue}.\\
\address{School of Mathematics and System Sciences, Xinjiang University,\\
Urumqi 830046, P. R. China}}

\begin{abstract}
In this note, we establish several interpolation inequalities in $\mathbb R^n$ in the Lebesgue spaces and Morrey spaces. By using the classical Calderon--Zygmund decomposition, we will reprove that $L^{p}(\mathbb R^n)\cap\mathrm{BMO}(\mathbb R^n)\subset L^{q}(\mathbb R^n)$ for all $q$ with $p<q<\infty$, where $1\leq p<\infty$. We also reprove that there exists a constant $C(p,q,n)$ depending on $p,q,n$ such that the following inequality
\begin{equation*}
\|f\|_{L^q}\leq C(p,q,n)\cdot\big(\|f\|_{L^p}\big)^{p/q}\cdot\big(\|f\|_{\mathrm{BMO}}\big)^{1-p/q}
\end{equation*}
holds for all $f\in L^{p}(\mathbb R^n)\cap\mathrm{BMO}(\mathbb R^n)$ with $1\leq p<\infty$. Moreover, this embedding constant has the optimal growth order $q$ as $q\to\infty$, which was given by Chen--Zhu, and Kozono--Wadade. We will show that $L^{p,\kappa}(\mathbb R^n)\cap\mathrm{BMO}(\mathbb R^n)\subset L^{q,\kappa}(\mathbb R^n)$ for all $q$ with $p<q<\infty$, where $1\leq p<\infty$ and $0<\kappa<1$. Moreover, there exists a constant $\widetilde{C}(p,q,n)$ depending on $p,q,n$ such that
\begin{equation*}
\|f\|_{L^{q,\kappa}}\leq \widetilde{C}(p,q,n)\cdot\big(\|f\|_{L^{p,\kappa}}\big)^{p/q}\cdot\big(\|f\|_{\mathrm{BMO}}\big)^{1-p/q}
\end{equation*}
holds for all $f\in L^{p,\kappa}(\mathbb R^n)\cap\mathrm{BMO}(\mathbb R^n)$ with $1\leq p<\infty$ and $0<\kappa<1$. This embedding constant is shown to have the linear growth order as $q\to\infty$, that is, $\widetilde{C}(p,q,n)\leq C_n\cdot q$ with the constant $C_n$ depending only on the dimension $n$, when $q$ is large. As an application of the above results, some new bilinear estimates are also established, which can be used in the study of the global existence and regularity of weak solutions to elliptic and parabolic partial differential equations of the second order.
\end{abstract}


\end{frontmatter}

\section{Introduction}
\label{intro}
In this paper, the symbols $\mathbb R$ and $\mathbb N$ stand for the sets of all real numbers and natural numbers, respectively. Let $\mathbb R^n$ be the $n$-dimensional Euclidean space endowed with the Euclidean norm $|\cdot|$ and the Lebesgue measure $dx$. For any given $p\in[1,\infty)$, the space $L^p(\mathbb R^n)$ is defined as the set of all integrable functions $f$ on $\mathbb R^n$ such that
\begin{equation*}
\|f\|_{L^p}:=\bigg(\int_{\mathbb R^n}|f(x)|^p\,dx\bigg)^{1/p}<+\infty,
\end{equation*}
and the weak space $WL^{p}(\mathbb R^n)$ is defined as the set of all measurable functions $f$ on $\mathbb R^n$ such that
\begin{equation*}
\|f\|_{WL^p}:=\sup_{\lambda>0}\lambda\cdot \big|\big\{x\in\mathbb R^n:|f(x)|>\lambda\big\}\big|^{1/p}<+\infty.
\end{equation*}
In the sequel, for a measurable set $E\subset\mathbb R^n$, the $n$-dimensional Lebesgue measure of $E$ is denoted by $|E|$. For positive $\lambda$, the function $d_{f}(\lambda):=\big|\big\{x\in\mathbb R^n:|f(x)|>\lambda\big\}\big|$ is called the distribution function of $f$. For the case $p=\infty$, $WL^\infty(\mathbb R^n)$ will be taken to mean $L^\infty(\mathbb R^n)$, which is defined as the set of all measurable functions $f$ on $\mathbb R^n$ such that
\begin{equation*}
\|f\|_{L^\infty}:=\underset{x\in\mathbb R^n}{\mbox{ess\,sup}}\,|f(x)|<+\infty.
\end{equation*}
We say that a locally integrable function $f$ on $\mathbb R^n$ has bounded mean oscillation, if
\begin{equation*}
\|f\|_{\mathrm{BMO}}:=\sup_{Q\subset\mathbb R^n}\frac{1}{|Q|}\int_{Q}|f(x)-f_{Q}|\,dx<+\infty,
\end{equation*}
where the supremum is taken over all cubes $Q$ in $\mathbb R^n$ with sides parallel to the coordinate axes, and $f_{Q}$ denotes the integral mean of $f$ over the cube $Q$, namely
\begin{equation*}
f_{Q}:=\frac{1}{|Q|}\int_{Q}f(y)\,dy.
\end{equation*}
Modulo constants, $\|\cdot\|_{\mathrm{BMO}}$ defines a norm and the Banach space of such functions is denoted by $\mathrm{BMO}(\mathbb R^n)$.
\begin{rem}
It is well known that for any $1\leq p<\infty$, $L^p(\mathbb R^n)\subset WL^{p}(\mathbb R^n)$, and this inclusion is strict. Let $F(x):=|x|^{-n/p}$ defined on $\mathbb R^n$ with the usual Lebesgue measure. It is obvious that $F$ is not in $L^p(\mathbb R^n)$ but $F$ is in $WL^p(\mathbb R^n)$ with
$\|F\|_{WL^p}=(\nu_n)^{1/p}$, where $\nu_n$ is the measure of the unit ball in $\mathbb R^n$.
\end{rem}

\begin{rem}
It is well known that $L^\infty(\mathbb R^n)\subset \mathrm{BMO}(\mathbb R^n)$, and this inclusion is strict. The typical example of a function that is in $\mathrm{BMO}(\mathbb R^n)$ but not in $L^\infty(\mathbb R^n)$ is $F(x):=\log |x|$ on $\mathbb R^n$. Furthermore, it is not difficult to see that if $\delta\geq0$, then $(\log|x|)^{\delta}\in \mathrm{BMO}(\mathbb R^n)$ if and only if $\delta\leq1$. The space $\mathrm{BMO}(\mathbb R^n)$ arises naturally as a substitute for the space $L^\infty(\mathbb R^n)$ in certain limiting cases. For example, the Calderon--Zygmund operator with standard kernel is known to be bounded on $L^p(\mathbb R^n)$(for $1<p<\infty$) but not bounded in $L^p(\mathbb R^n)$ for $p=\infty$. Actually, it is bounded from $L^\infty(\mathbb R^n)$ into $\mathrm{BMO}(\mathbb R^n)$. The Riesz potential operator of order $\alpha$ is known to be bounded from $L^p(\mathbb R^n)$ into $L^q(\mathbb R^n)$ when $1<p<q<\infty$ and $1/q=1/p-\alpha/n$, but not bounded from $L^p(\mathbb R^n)$ into $L^{\infty}(\mathbb R^n)$ when $p=n/{\alpha}$. In this case, it is actually bounded from $L^{n/{\alpha}}(\mathbb R^n)$ into $\mathrm{BMO}(\mathbb R^n)$(see \cite{duoand,stein}).
\end{rem}

Let us now recall the definitions of the Morrey space and weak Morrey space (see \cite{adams1}). Let $1\leq p<\infty$ and $0\leq\kappa\leq1$. We say that a locally integrable function $f$ on $\mathbb R^n$ belongs to the Morrey space $L^{p,\kappa}(\mathbb R^n)$, if
\begin{equation*}
\|f\|_{L^{p,\kappa}}:=\sup_{Q\subset\mathbb R^n}\bigg(\frac{1}{|Q|^{\kappa}}\int_{Q}|f(y)|^p\,dy\bigg)^{1/p}<+\infty,
\end{equation*}
where the supremum is taken over all cubes $Q$ in $\mathbb R^n$. We say that a measurable function $f$ on $\mathbb R^n$ belongs to the weak Morrey space $WL^{p,\kappa}(\mathbb R^n)$, if
\begin{equation*}
\|f\|_{WL^{p,\kappa}}:=\sup_{Q\subset\mathbb R^n}\sup_{\lambda>0}
\frac{1}{|Q|^{\kappa/p}}\lambda\cdot\big|\big\{y\in Q:|f(y)|>\lambda\big\}\big|^{1/p}<+\infty,
\end{equation*}
where the supremum is taken over all cubes $Q\subset\mathbb R^n$ and all $\lambda>0$.
\begin{rem}
When $\kappa=0$, then we have $L^{p,0}(\mathbb R^n)=L^p(\mathbb R^n)$ and $WL^{p,0}(\mathbb R^n)=WL^p(\mathbb R^n)$ by definition, and when $\kappa=1$, then we have $L^{p,1}(\mathbb R^n)=L^\infty(\mathbb R^n)$ by the Lebesgue differentiation theorem.
\end{rem}

\begin{rem}
Let $0\leq\kappa<1$. It is well known that the Hardy--Littlewood maximal operator $M$ is bounded from $L^{1,\kappa}(\mathbb R^n)$ to $WL^{1,\kappa}(\mathbb R^n)$ by \cite{chia}, but it is not bounded on $L^{1,\kappa}(\mathbb R^n)$ by \cite{nakai}. This suggests that there exists a function $\mathcal{F}\in WL^{1,\kappa}(\mathbb R^n)$ such that $\mathcal{F}\notin L^{1,\kappa}(\mathbb R^n)$. A constructive proof of this fact can be seen in \cite{gun}. Moreover, it can be shown that there exists a function $\mathcal{F}\in WL^{p,\kappa}(\mathbb R^n)\setminus L^{p,\kappa}(\mathbb R^n)$ for $1\leq p<\infty$ and $0<\kappa<1$ (see \cite[Theorem 1.2]{gun}, for instance).
\end{rem}
In this work, we shall be interested in various interpolation inequalities in the Lebesgue spaces and Morrey spaces. The main results of this paper can be stated as follows.

Let $1\leq p<\infty$ and $n\in\mathbb N$.
\begin{itemize}
  \item $L^{p}(\mathbb R^n)\cap L^\infty(\mathbb R^n)\subset L^{q}(\mathbb R^n)$ for all $q$ with $p<q<\infty$.
  \item $WL^{p}(\mathbb R^n)\cap L^\infty(\mathbb R^n)\subset L^{q}(\mathbb R^n)$ for all $q$ with $p<q<\infty$.
  \item $L^{p}(\mathbb R^n)\cap\mathrm{BMO}(\mathbb R^n)\subset L^{q}(\mathbb R^n)$ for all $q$ with $p<q<\infty$.
\end{itemize}
The third result was obtained by Chen and Zhu in \cite[Theorem 2]{Chen}, by using the technique of the non-increasing rearrangement, and then extended by many authors, see, for example, \cite[Theorems 2.1 and 2.2]{Kozono2} and \cite[Theorem 1.1]{wadade}. A new, more elementary proof of this result, avoiding the non-increasing rearrangement, is given below.

Let $1\leq p<\infty$, $0<\kappa<1$ and $n\in\mathbb N$.
\begin{itemize}
  \item $L^{p,\kappa}(\mathbb R^n)\cap L^\infty(\mathbb R^n)\subset L^{q,\kappa}(\mathbb R^n)$ for all $q$ with $p<q<\infty$.
  \item $WL^{p,\kappa}(\mathbb R^n)\cap L^\infty(\mathbb R^n)\subset L^{q,\kappa}(\mathbb R^n)$ for all $q$ with $p<q<\infty$.
  \item $L^{p,\kappa}(\mathbb R^n)\cap\mathrm{BMO}(\mathbb R^n)\subset L^{q,\kappa}(\mathbb R^n)$ for all $q$ with $p<q<\infty$.
\end{itemize}
Moreover, the embedding constants in the inclusion relations above are computed explicitly. The sharpness of these interpolation inequalities (optimal growth rate as $q\to\infty$) is also investigated. As a consequence of our main results, some new bilinear estimates are also considered, which is in connection with smoothness properties (H\"{o}lder-continuity) of weak solutions of nonlinear elliptic and parabolic equations.

\section{Main results}
\label{sec:1}
\subsection{Interpolation inequalities in the Lebesgue spaces}
In this section, we are concerned with some interpolation inequalities in the Lebesgue spaces. The first two results are known in the literature and can be found in \cite{Grafakos2} (see also \cite{li}). Here we give the proof for the reader's convenience.
\label{sec2}
\begin{thm}\label{thma}
Let $1\leq p<\infty$ and $n\in\mathbb N$. If $f\in L^p(\mathbb R^n)\cap L^\infty(\mathbb R^n)$, then for every $q$ with $p<q<\infty$, we have $f\in L^q(\mathbb R^n)$, and the following inequality holds.
\begin{equation*}
\|f\|_{L^q}\leq 1\cdot\big(\|f\|_{L^p}\big)^{p/q}\cdot\big(\|f\|_{L^\infty}\big)^{1-p/q}.
\end{equation*}
\end{thm}
\begin{proof}
The proof is very simple. Notice that $q-p>0$. It is easy to see that
\begin{equation*}
\begin{split}
\|f\|_{L^q}&=\bigg(\int_{\mathbb R^n}|f(x)|^p\cdot|f(x)|^{q-p}\,dx\bigg)^{1/q}\\
&\leq \bigg(\int_{\mathbb R^n}|f(x)|^p\,dx\bigg)^{1/q}\cdot\big(\|f\|_{L^\infty}\big)^{1-p/q}\\
&=\big(\|f\|_{L^p}\big)^{p/q}\cdot\big(\|f\|_{L^\infty}\big)^{1-p/q},
\end{split}
\end{equation*}
as desired.
\end{proof}

\begin{thm}\label{thmb}
Let $1\leq p<\infty$ and $n\in\mathbb N$. If $f\in WL^p(\mathbb R^n)\cap L^\infty(\mathbb R^n)$, then for every $q$ with $p<q<\infty$, we have $f\in L^q(\mathbb R^n)$, and the following inequality holds.
\begin{equation*}
\|f\|_{L^q}\leq\left(\frac{q}{q-p}\right)^{1/q}\cdot\big(\|f\|_{WL^p}\big)^{p/q}\cdot\big(\|f\|_{L^\infty}\big)^{1-p/q}.
\end{equation*}
\end{thm}
\begin{proof}
First we estimate the $L^q$-norm of $f$ in terms of the distribution function.
\begin{equation}\label{WLp}
\begin{split}
\|f\|^q_{L^q}&=\int_{\mathbb R^n}|f(x)|^q\,dx
=\int_0^{\infty}\big|\big\{x\in\mathbb R^n:|f(x)|>\lambda\big\}\big|q\lambda^{q-1}\,d\lambda\\
&=\bigg(\int_0^{\|f\|_{L^\infty}}+\int_{\|f\|_{L^\infty}}^{\infty}\bigg)\big|\big\{x\in\mathbb R^n:|f(x)|>\lambda\big\}\big|q\lambda^{q-1}\,d\lambda.
\end{split}
\end{equation}
Note that
\begin{equation*}
\big|\big\{x\in\mathbb R^n:|f(x)|>\lambda\big\}\big|=0,\quad \mathrm{when}~ \lambda\geq\|f\|_{L^\infty}.
\end{equation*}
Hence, the second integral in \eqref{WLp} is equal to $0$. Let us now estimate the first integral in \eqref{WLp}. Since $f\in WL^p(\mathbb R^n)$ with $1\leq p<\infty$, we can deduce that
\begin{equation*}
\begin{split}
&\int_0^{\|f\|_{L^\infty}}\big|\big\{x\in\mathbb R^n:|f(x)|>\lambda\big\}\big|q\lambda^{q-1}\,d\lambda\\
&\leq \int_0^{\|f\|_{L^\infty}}\left(\frac{\|f\|_{WL^p}}{\lambda}\right)^pq\lambda^{q-1}\,d\lambda\\
&=\frac{q}{q-p}\cdot\big(\|f\|_{WL^p}\big)^p\cdot\big(\|f\|_{L^\infty}\big)^{q-p}.
\end{split}
\end{equation*}
Observe that the above integral converges since $q-p>0$. From this, the desired result follows immediately.
\end{proof}

\begin{rem}
It should be pointed out that both the embedding constant $1$ in Theorem \ref{thma} and the embedding constant $[q/{(q-p)}]^{1/q}$ in Theorem \ref{thmb} are optimal.
\end{rem}

The following result is well known in harmonic analysis (see \cite{John}).
\begin{thm}[The John--Nirenberg theorem]\label{JN}
Let $n\in\mathbb N$ and $f\in\mathrm{BMO}(\mathbb R^n)$. There exist two positive constants $C_1$ and $C_2$, depending only on the dimension $n$, such that for any cube $\mathcal{Q}$ in $\mathbb R^n$ and any $\lambda>0$,
\begin{equation*}
\big|\big\{x\in \mathcal{Q}:|f(x)-f_{\mathcal{Q}}|>\lambda\big\}\big|\leq C_1|\mathcal{Q}|\exp\bigg\{-\frac{C_2\lambda}{\|f\|_{\mathrm{BMO}}}\bigg\}.
\end{equation*}
\end{thm}
More specifically, we can take $C_1=e$ and $C_2=(2^n e)^{-1}$. See, for example, \cite[Theorem 7.1.6]{Grafakos}.

The following result was established by Chen and Zhu in \cite{Chen}, and then extended by Kozono and Wadade \cite{Kozono2}, and Wadade \cite{wadade}.
However, our method is different from theirs.
\begin{thm}\label{thmc}
Let $1\leq p<\infty$ and $n\in\mathbb N$. If $f\in L^p(\mathbb R^n)\cap \mathrm{BMO}(\mathbb R^n)$, then for every $q$ with $p<q<\infty$, we have $f\in L^q(\mathbb R^n)$, and the following inequality holds.
\begin{equation*}
\|f\|_{L^q}\leq C(p,q,n)\cdot\big(\|f\|_{L^p}\big)^{p/q}\cdot\big(\|f\|_{\mathrm{BMO}}\big)^{1-p/q},
\end{equation*}
where
\begin{equation*}
C(p,q,n):=\bigg[\frac{q}{q-p}+2^{\frac{nq}{p}}+2^{nq}\Gamma(q+1)e^{a+q+1}\bigg]^{1/q} \quad \mathrm{with}~~ a=\frac{1}{2^{\frac{n}{p'}}\cdot e}.
\end{equation*}
\end{thm}

\begin{proof}
We give an alternative proof of this result, which is based on the classical Calderon--Zygmund decomposition and Theorem \ref{JN}.
By the assumption $f\in L^p(\mathbb R^n)\cap \mathrm{BMO}(\mathbb R^n)$, we know that
\begin{equation*}
0<\|f\|_{\mathrm{BMO}}<+\infty.
\end{equation*}
For $1\leq p<\infty$, we apply the Calderon--Zygmund decomposition of $|f|^p$ at height $\sigma:=\|f\|^p_{\mathrm{BMO}}$ to obtain a collection of disjoint cubes $\{Q_k\}_{k\geq1}$ such that the following properties hold.
\begin{equation*}
\begin{cases}
\mathbb R^n=F\cup\Omega,~ F\cap\Omega=\O;&\\
|f(x)|^p\leq\sigma,~~a.e.~ x\in F;&\\
\Omega=\bigcup_{k=1}^{\infty}Q_k,~\sigma<\frac{1}{|Q_k|}\int_{Q_k}|f(x)|^p\,dx\leq 2^n\sigma, ~\forall\,k\in \mathbb{N}.
\end{cases}
\end{equation*}
We now estimate the $L^q$-norm of $f$ in terms of the distribution function with $p<q<\infty$.
\begin{equation}\label{bmo1}
\begin{split}
\|f\|^q_{L^q}&=\int_{\mathbb R^n}|f(x)|^q\,dx
=\int_0^{\infty}\big|\big\{x\in\mathbb R^n:|f(x)|>\lambda\big\}\big|q\lambda^{q-1}\,d\lambda\\
&=\int_0^{\infty}\big|\big\{x\in F:|f(x)|>\lambda\big\}\big|q\lambda^{q-1}\,d\lambda\\
&+\int_0^{\infty}\big|\big\{x\in\Omega:|f(x)|>\lambda\big\}\big|q\lambda^{q-1}\,d\lambda\\
&:=\mathrm{I+II}.
\end{split}
\end{equation}
Let us first give the estimation of the term I. In view of the above construction, we have $|f(x)|\leq \|f\|_{\mathrm{BMO}}$, $a.e.~ x\in F$, which yields
\begin{equation*}
\big|\big\{x\in F:|f(x)|>\lambda\big\}\big|=0
\end{equation*}
whenever $\lambda\geq \|f\|_{\mathrm{BMO}}$. Consequently,
\begin{equation*}
\begin{split}
\mathrm{I}&=\bigg(\int_0^{\|f\|_{\mathrm{BMO}}}+\int_{\|f\|_{\mathrm{BMO}}}^{\infty}\bigg)
\big|\big\{x\in F:|f(x)|>\lambda\big\}\big|q\lambda^{q-1}\,d\lambda\\
&=\int_0^{\|f\|_{\mathrm{BMO}}}\big|\big\{x\in F:|f(x)|>\lambda\big\}\big|q\lambda^{q-1}\,d\lambda.
\end{split}
\end{equation*}
Since $f\in L^p(\mathbb R^n)$ with $p<q$, the last expression is bounded by
\begin{equation*}
\begin{split}
&\int_0^{\|f\|_{\mathrm{BMO}}}\left(\frac{\|f\|_{L^p}}{\lambda}\right)^pq\lambda^{q-1}\,d\lambda\\
&=\frac{q}{q-p}\big(\|f\|_{L^p}\big)^p\cdot\big(\|f\|_{\mathrm{BMO}}\big)^{q-p}.
\end{split}
\end{equation*}
Now we turn our attention to the other term II. By using H\"older's inequality, we can see that when $x\in Q_k$, $k\in \mathbb{N}$,
\begin{equation*}
\begin{split}
|f(x)|&\leq|f(x)-f_{Q_k}|+\frac{1}{|Q_k|}\int_{Q_k}|f(x)|\,dx\\
&\leq|f(x)-f_{Q_k}|+\bigg(\frac{1}{|Q_k|}\int_{Q_k}|f(x)|^p\,dx\bigg)^{1/p}\\
&\leq |f(x)-f_{Q_k}|+2^{n/p}\|f\|_{\mathrm{BMO}}.
\end{split}
\end{equation*}
Hence
\begin{equation}\label{termII}
\begin{split}
\mathrm{II}&\leq\int_0^{\infty}\sum_{k=1}^\infty\big|\big\{x\in Q_k:|f(x)|>\lambda\big\}\big|q\lambda^{q-1}\,d\lambda\\
&\leq\int_0^{\infty}\sum_{k=1}^\infty
\big|\big\{x\in Q_k:|f(x)-f_{Q_k}|>\lambda-2^{\frac{n}{p}}\|f\|_{\mathrm{BMO}}\big\}\big|q\lambda^{q-1}\,d\lambda\\
&=\bigg(\int_0^{2^{\frac{n}{p}}\|f\|_{\mathrm{BMO}}}+\int_{2^{\frac{n}{p}}\|f\|_{\mathrm{BMO}}}^\infty\bigg)\sum_{k=1}^\infty
\big|\big\{x\in Q_k:|f(x)-f_{Q_k}|>\lambda-2^{\frac{n}{p}}\|f\|_{\mathrm{BMO}}\big\}\big|q\lambda^{q-1}\,d\lambda.
\end{split}
\end{equation}
Observe that if $\lambda<2^{n/p}\|f\|_{\mathrm{BMO}}$, then for each cube $Q_k$($k\in \mathbb{N}$),
\begin{equation*}
\big\{x\in Q_k:|f(x)-f_{Q_k}|>\lambda-2^{\frac{n}{p}}\|f\|_{\mathrm{BMO}}\big\}=Q_k.
\end{equation*}
In this case, we then have
\begin{equation}\label{111}
\begin{split}
&\sum_{k=1}^\infty\big|\big\{x\in Q_k:|f(x)-f_{Q_k}|>\lambda-2^{\frac{n}{p}}\|f\|_{\mathrm{BMO}}\big\}\big|\\
&=\sum_{k=1}^\infty|Q_k|<\frac{1}{\sigma}\sum_{k=1}^\infty\int_{Q_k}|f(x)|^p\,dx\\
&=\frac{1}{\sigma}\int_{\bigcup_{k=1}^{\infty}Q_k}|f(x)|^p\,dx
\leq\left(\frac{\|f\|_{L^p}}{\|f\|_{\mathrm{BMO}}}\right)^p.
\end{split}
\end{equation}
From \eqref{111}, it follows that the first summand in \eqref{termII} is bounded by
\begin{equation*}
\int_0^{2^{\frac{n}{p}}\|f\|_{\mathrm{BMO}}}\left(\frac{\|f\|_{L^p}}{\|f\|_{\mathrm{BMO}}}\right)^pq\lambda^{q-1}\,d\lambda
=2^{\frac{nq}{p}}\big(\|f\|_{L^p}\big)^p\cdot\big(\|f\|_{\mathrm{BMO}}\big)^{q-p}.
\end{equation*}
On the other hand, by using Theorem \ref{JN}, the second summand in \eqref{termII} can be estimated as follows.
\begin{equation*}
\begin{split}
&\int_{2^{\frac{n}{p}}\|f\|_{\mathrm{BMO}}}^\infty\sum_{k=1}^\infty
\big|\big\{x\in Q_k:|f(x)-f_{Q_k}|>\lambda-2^{\frac{n}{p}}\|f\|_{\mathrm{BMO}}\big\}\big|q\lambda^{q-1}\,d\lambda\\
&=\sum_{k=1}^\infty\int_0^{\infty}
\big|\big\{x\in Q_k:|f(x)-f_{Q_k}|>\nu\big\}\big|q\Big(\nu+2^{\frac{n}{p}}\|f\|_{\mathrm{BMO}}\Big)^{q-1}\,d\nu\\
&\leq eq\sum_{k=1}^\infty|Q_k|\times\int_0^{\infty}\exp\bigg\{-\frac{\nu}{2^n e\|f\|_{\mathrm{BMO}}}\bigg\}
\cdot\Big(\nu+2^{\frac{n}{p}}\|f\|_{\mathrm{BMO}}\Big)^{q-1}\,d\nu\\
&\leq eq\left(\frac{\|f\|_{L^p}}{\|f\|_{\mathrm{BMO}}}\right)^p\int_0^{\infty}e^{-\mu}
\cdot\Big(2^n e\|f\|_{\mathrm{BMO}}\cdot\mu+2^{\frac{n}{p}}\|f\|_{\mathrm{BMO}}\Big)^{q-1}\Big(2^n e\|f\|_{\mathrm{BMO}}\Big)\,d\mu\\
&=eq\big(2^ne\big)^q\big(\|f\|_{L^p}\big)^p\cdot\big(\|f\|_{\mathrm{BMO}}\big)^{q-p}
\int_0^{\infty}e^{-\mu}\left(\mu+\frac{1}{2^{\frac{n}{p'}}\cdot e}\right)^{q-1}d\mu,
\end{split}
\end{equation*}
where in the last inequality we have used the estimate \eqref{111}. Moreover, a direct calculation shows that
\begin{equation*}
\begin{split}
\int_0^{\infty}e^{-\mu}\left(\mu+\frac{1}{2^{\frac{n}{p'}}\cdot e}\right)^{q-1}d\mu
&=\int_{a}^{\infty}e^{-(u-a)}u^{q-1}du~\Big(a=\frac{1}{2^{\frac{n}{p'}}\cdot e}\Big)\\
&\leq e^a\Gamma(q),
\end{split}
\end{equation*}
where $\Gamma(\cdot)$ denotes the usual gamma function. Summing up the above estimates, we conclude that
\begin{equation*}
\|f\|^q_{L^q}\leq \mathcal{A}(p,q,n)\big(\|f\|_{L^p}\big)^p\cdot\big(\|f\|_{\mathrm{BMO}}\big)^{q-p}.
\end{equation*}
Here
\begin{equation*}
\begin{split}
\mathcal{A}(p,q,n)&=\frac{q}{q-p}+2^{\frac{nq}{p}}+2^{nq}q\Gamma(q)e^{a+q+1}\\
&=\frac{q}{q-p}+2^{\frac{nq}{p}}+2^{nq}\Gamma(q+1)e^{a+q+1} \quad \mathrm{with}~~ a=\frac{1}{2^{\frac{n}{p'}}\cdot e}.
\end{split}
\end{equation*}
This implies our desired estimate.
\end{proof}

Some remarks are in order.
\begin{rem}
$(1)$ The constant $C(p,q,n)$ in Theorem \ref{thmc} has the form
\begin{equation*}
C(p,q,n)=\mathcal{O}(q),\quad \mbox{as}~ q\to\infty,
\end{equation*}
that is, $C(p,q,n)\leq C_n\cdot q$ when $q$ is large. Here $C_n$ is an absolute constant independent of $q$. Indeed, it is easy to see that
\begin{equation*}
\left(\frac{q}{q-p}\right)^{1/q}=\mathcal{O}(1),\quad \frac{a+q+1}{q}=\mathcal{O}(1)~(q\to\infty),
\quad \& \quad \big(2^{\frac{nq}{p}}\big)^{1/q}=2^{\frac{n}{p}}\leq 2^n.
\end{equation*}
In addition, with Stirling's formula, we further have
\begin{equation*}
\Gamma(q+1)^{1/q}=\mathcal{O}(q),\quad \mbox{as}~ q\to\infty,
\end{equation*}
from which the desired result follows immediately. Furthermore, it can be shown that the growth order $q$ in Theorem \ref{thmc} is sharp as $q\to\infty$(see \cite[Theorem 2.2 and Remark]{Kozono2}).

$(2)$ It is not clear to us whether Theorem \ref{thmc} can be improved, we do not know whether the conclusion of Theorem \ref{thmc} still holds for $f\in WL^p(\mathbb R^n)\cap \mathrm{BMO}(\mathbb R^n)$ with $1\leq p<\infty$.

$(3)$ Let $1\leq r<\infty$. An important consequence of Theorem \ref{thmc} is the following:
\begin{equation*}
\|F\cdot G\|_{L^r}\leq C(r,n)\Big(\|F\|_{L^r}\|G\|_{\mathrm{BMO}}\cdot\|G\|_{L^r}\|F\|_{\mathrm{BMO}}\Big)^{1/2}
\end{equation*}
holds for all $F,G\in L^r(\mathbb R^n)\cap \mathrm{BMO}(\mathbb R^n)$. In fact, by using H\"older's inequality and Theorem \ref{thmc}, we obtain
\begin{equation*}
\begin{split}
\|F\cdot G\|_{L^r}&\leq\|F\|_{L^{2r}}\|G\|_{L^{2r}}\\
&\leq C(r,n)\Big(\|F\|_{L^r}^{1/2}\|F\|_{\mathrm{BMO}}^{1/2}\cdot\|G\|_{L^r}^{1/2}\|G\|_{\mathrm{BMO}}^{1/2}\Big)\\
&=C(r,n)\Big(\|F\|_{L^r}\|G\|_{\mathrm{BMO}}\cdot\|G\|_{L^r}\|F\|_{\mathrm{BMO}}\Big)^{1/2},
\end{split}
\end{equation*}
as desired. From the elementary inequality $2(ab)^{1/2}\leq a+b$, $a,b>0$, we have the following bilinear estimates in $L^r(\mathbb R^n)\cap\mathrm{BMO}(\mathbb R^n)$.
\begin{equation*}
\|F\cdot G\|_{L^r}\leq C(r,n)\Big(\|F\|_{L^r}\|G\|_{\mathrm{BMO}}+\|G\|_{L^r}\|F\|_{\mathrm{BMO}}\Big).
\end{equation*}
These bilinear estimates were proved by Kozono and Taniuchi in \cite{Kozono}. Such kind of inequalities are used to extend some results on uniqueness and regularity of weak solutions to the Navier--Stokes equations.
\end{rem}

\subsection{Interpolation inequalities in the Morrey spaces}
\label{sec3}
In this section, we are concerned with several interpolation inequalities in the Morrey spaces.
\begin{thm}\label{Thm1}
Let $1\leq p<\infty$, $0<\kappa<1$ and $n\in\mathbb N$. If $f\in L^{p,\kappa}(\mathbb R^n)\cap L^\infty(\mathbb R^n)$, then for every $q$ with $p<q<\infty$, we have $f\in L^{q,\kappa}(\mathbb R^n)$, and the following inequality holds.
\begin{equation*}
\|f\|_{L^{q,\kappa}}\leq 1\cdot\big(\|f\|_{L^{p,\kappa}}\big)^{p/q}\cdot\big(\|f\|_{L^\infty}\big)^{1-p/q}.
\end{equation*}
\end{thm}
\begin{proof}
Using the same strategy as in Theorem \ref{thma}, for any cube $Q$ in $\mathbb R^n$, we obtain
\begin{equation*}
\begin{split}
\bigg(\frac{1}{|Q|^{\kappa}}\int_{Q}|f(y)|^q\,dy\bigg)^{1/q}
&=\bigg(\frac{1}{|Q|^{\kappa}}\int_{Q}|f(y)|^p\cdot|f(y)|^{q-p}\,dy\bigg)^{1/q}\\
&\leq\bigg(\frac{1}{|Q|^{\kappa}}\int_{Q}|f(y)|^p\,dy\bigg)^{1/q}\cdot\big(\|f\|_{L^\infty}\big)^{1-p/q}\\
&\leq\big(\|f\|_{L^{p,\kappa}}\big)^{p/q}\cdot\big(\|f\|_{L^\infty}\big)^{1-p/q},
\end{split}
\end{equation*}
where we have used the fact that $q-p>0$. By taking the supremum over all cubes $Q$ in $\mathbb R^n$, we complete the proof of Theorem \ref{Thm1}.
\end{proof}

\begin{thm}\label{Thm2}
Let $1\leq p<\infty$, $0<\kappa<1$ and $n\in\mathbb N$. If $f\in WL^{p,\kappa}(\mathbb R^n)\cap L^\infty(\mathbb R^n)$, then for every $q$ with $p<q<\infty$, we have $f\in L^{q,\kappa}(\mathbb R^n)$, and the following inequality holds.
\begin{equation*}
\|f\|_{L^{q,\kappa}}\leq\left(\frac{q}{q-p}\right)^{1/q}\cdot\big(\|f\|_{WL^{p,\kappa}}\big)^{p/q}\cdot\big(\|f\|_{L^\infty}\big)^{1-p/q}.
\end{equation*}
\end{thm}

\begin{proof}
Using the same strategy as in Theorem \ref{thmb}, for any given cube $Q$ in $\mathbb R^n$, we have
\begin{equation}\label{WLpk}
\begin{split}
&\frac{1}{|Q|^{\kappa}}\int_{Q}|f(y)|^q\,dy\\
&=\frac{1}{|Q|^{\kappa}}\int_0^{\infty}\big|\big\{y\in Q:|f(y)|>\lambda\big\}\big|q\lambda^{q-1}\,d\lambda\\
&=\frac{1}{|Q|^{\kappa}}\bigg(\int_0^{\|f\|_{L^\infty}}+\int_{\|f\|_{L^\infty}}^{\infty}\bigg)
\big|\big\{y\in Q:|f(y)|>\lambda\big\}\big|q\lambda^{q-1}\,d\lambda.
\end{split}
\end{equation}
Note that for any cube $Q\subset\mathbb R^n$,
\begin{equation*}
\big|\big\{y\in Q:|f(y)|>\lambda\big\}\big|=0,\quad \mathrm{when}~\lambda\geq\|f\|_{L^\infty}.
\end{equation*}
Hence, the second integral in \eqref{WLpk} is equal to $0$. Let us now estimate the first integral in \eqref{WLpk}. Since $f\in WL^{p,\kappa}(\mathbb R^n)$ with $1\leq p<\infty$ and $0<\kappa<1$, we thus obtain
\begin{equation*}
\begin{split}
&\frac{1}{|Q|^{\kappa}}\int_0^{\|f\|_{L^\infty}}\big|\big\{y\in Q:|f(y)|>\lambda\big\}\big|q\lambda^{q-1}\,d\lambda\\
&\leq\frac{1}{|Q|^{\kappa}}\int_0^{\|f\|_{L^\infty}}\bigg(\frac{|Q|^{\kappa/p}\|f\|_{WL^{p,\kappa}}}{\lambda}\bigg)^pq\lambda^{q-1}\,d\lambda\\
&=\frac{q}{q-p}\cdot\big(\|f\|_{WL^{p,\kappa}}\big)^p\cdot\big(\|f\|_{L^\infty}\big)^{q-p}.
\end{split}
\end{equation*}
Observe also that the above integral converges since $q-p>0$. Taking $q$-th roots of both sides in \eqref{WLpk} and the supremum over all cubes $Q$ in $\mathbb R^n$, we finish the proof of Theorem \ref{Thm2}.
\end{proof}
\begin{rem}
Here we find explicit values for these constants. It can be shown again that the embedding constant $1$ in Theorem \ref{Thm1} and the embedding constant $[q/{(q-p)}]^{1/q}$ in Theorem \ref{Thm2} are best possible.
\end{rem}
Based on certain appropriate Calderon--Zygmund decomposition, we can prove the following result, which is a generalization of Theorem \ref{thmc}.
\begin{thm}\label{Thm3}
Let $1\leq p<\infty$, $0<\kappa<1$ and $n\in\mathbb N$. If $f\in L^{p,\kappa}(\mathbb R^n)\cap\mathrm{BMO}(\mathbb R^n)$, then for every $q$ with $p<q<\infty$, we have $f\in L^{q,\kappa}(\mathbb R^n)$, and the following inequality holds.
\begin{equation*}
\|f\|_{L^{q,\kappa}}\leq \widetilde{C}(p,q,n)\cdot\big(\|f\|_{L^{p,\kappa}}\big)^{p/q}\cdot\big(\|f\|_{\mathrm{BMO}}\big)^{1-p/q},
\end{equation*}
where
\begin{equation*}
\widetilde{C}(p,q,n):=\bigg[\frac{q}{q-p}+2^{\frac{nq}{p}}+2^{(n+1)q}\Gamma(q+1)e^{b+q+1}\bigg]^{1/q} \quad \mathrm{with}~~ b=\frac{1}{2^{\frac{n}{p'}+1}\cdot e}.
\end{equation*}
\end{thm}

\begin{proof}
By the assumption $f\in L^{p,\kappa}(\mathbb R^n)\cap\mathrm{BMO}(\mathbb R^n)$, we know that
\begin{equation*}
0<\|f\|_{\mathrm{BMO}}<+\infty.
\end{equation*}
For any fixed cube $Q\subset\mathbb R^n$, it suffices to prove that
\begin{equation}\label{wanghua12}
\frac{1}{|Q|^{\kappa}}\int_{Q}|f(y)|^q\,dy\leq C\big(\|f\|_{L^{p,\kappa}}\big)^p\cdot\big(\|f\|_{\mathrm{BMO}}\big)^{q-p}
\end{equation}
holds for some positive constant $C$. Notice that $f\cdot\chi_{Q}\in L^p(\mathbb R^n)$ for such $Q$ and $1\leq p<\infty$. Here we denote the characteristic function of $Q$ by $\chi_Q$. We now apply the Calderon--Zygmund decomposition of $\mathcal{F}:=|f\cdot\chi_{Q}|^p$ at height $\sigma:=\|f\|^p_{\mathrm{BMO}}$ to obtain a collection of disjoint cubes $\{\mathcal{Q}_k\}_{k\geq1}$ such that the following properties hold.
\begin{equation*}
\begin{cases}
\mathbb R^n=F\cup\Omega, ~F\cap\Omega=\O;&\\
\big|f(y)\cdot\chi_{Q}(y)\big|^p\leq\sigma,~~a.e.~ y\in F;&\\
\Omega=\bigcup_{k=1}^{\infty}\mathcal{Q}_k,~\sigma<\frac{1}{|\mathcal{Q}_k|}\int_{\mathcal{Q}_k}\big|f(y)\cdot\chi_{Q}(y)\big|^p\,dy\leq 2^n\sigma,~\forall\,k\in \mathbb{N}.
\end{cases}
\end{equation*}
According to the above Calderon--Zygmund decomposition, the left-hand side in \eqref{wanghua12} can be estimated as follows.
\begin{equation}\label{bmo2}
\begin{split}
\frac{1}{|Q|^{\kappa}}\int_{Q}|f(y)|^q\,dy
&=\frac{1}{|Q|^{\kappa}}\int_0^{\infty}\big|\big\{y\in Q:|f(y)|>\lambda\big\}\big|q\lambda^{q-1}\,d\lambda\\
&=\frac{1}{|Q|^{\kappa}}\int_0^{\infty}\big|\big\{y\in Q\cap F:|f(y)|>\lambda\big\}\big|q\lambda^{q-1}\,d\lambda\\
&+\frac{1}{|Q|^{\kappa}}\int_0^{\infty}\big|\big\{y\in Q\cap\Omega:|f(y)|>\lambda\big\}\big|q\lambda^{q-1}\,d\lambda\\
&:=\mathrm{III+IV}.
\end{split}
\end{equation}
To estimate the term $\mathrm{III}$, we consider the following two cases: $\lambda\geq \|f\|_{\mathrm{BMO}}$ and $0<\lambda<\|f\|_{\mathrm{BMO}}$.
Note that $|f(y)\cdot\chi_{Q}(y)|\leq\|f\|_{\mathrm{BMO}}$, $a.e.~ y\in F$. This gives
\begin{equation*}
\big|\big\{y\in F:|f(y)\cdot\chi_{Q}(y)|>\lambda\big\}\big|=0
\end{equation*}
whenever $\lambda\geq \|f\|_{\mathrm{BMO}}$, or equivalently,
\begin{equation*}
\big|\big\{y\in Q\cap F:|f(y)|>\lambda\big\}\big|=0.
\end{equation*}
Consequently,
\begin{equation*}
\begin{split}
\mathrm{III}&=\frac{1}{|Q|^{\kappa}}\bigg(\int_0^{\|f\|_{\mathrm{BMO}}}+\int_{\|f\|_{\mathrm{BMO}}}^{\infty}\bigg)
\big|\big\{y\in Q\cap F:|f(y)|>\lambda\big\}\big|q\lambda^{q-1}\,d\lambda\\
&=\frac{1}{|Q|^{\kappa}}\int_0^{\|f\|_{\mathrm{BMO}}}\big|\big\{y\in Q\cap F:|f(y)|>\lambda\big\}\big|q\lambda^{q-1}\,d\lambda\\
&\leq\frac{1}{|Q|^{\kappa}}\int_0^{\|f\|_{\mathrm{BMO}}}\big|\big\{y\in Q:|f(y)|>\lambda\big\}\big|q\lambda^{q-1}\,d\lambda.
\end{split}
\end{equation*}
Since $f\in L^{p,\kappa}(\mathbb R^n)$ with $1\leq p<q$ and $0<\kappa<1$, we see that the last expression is dominated by
\begin{equation*}
\begin{split}
&\frac{1}{|Q|^{\kappa}}\int_0^{\|f\|_{\mathrm{BMO}}}\bigg(\frac{|Q|^{\kappa/p}\|f\|_{L^{p,\kappa}}}{\lambda}\bigg)^pq\lambda^{q-1}\,d\lambda\\
&=\frac{q}{q-p}\big(\|f\|_{L^{p,\kappa}}\big)^p\cdot\big(\|f\|_{\mathrm{BMO}}\big)^{q-p}.
\end{split}
\end{equation*}
We now turn to deal with the term $\mathrm{IV}$. For an arbitrary point $y\in Q\cap \mathcal{Q}_k$, $k=1,2,\dots$, we have
\begin{equation*}
\begin{split}
\big|f(y)\big|&=\big|f(y)\cdot\chi_{Q}(y)\big|\\
&\leq\big|(f\chi_{Q})(y)-(f\chi_Q)_{\mathcal{Q}_k}\big|+\frac{1}{|\mathcal{Q}_k|}\int_{\mathcal{Q}_k}\big|f(y)\cdot\chi_{Q}(y)\big|\,dy\\
&\leq\big|(f\chi_{Q})(y)-(f\chi_Q)_{\mathcal{Q}_k}\big|+\bigg(\frac{1}{|\mathcal{Q}_k|}\int_{\mathcal{Q}_k}\big|f(y)\cdot\chi_{Q}(y)\big|^p\,dy\bigg)^{1/p}\\
&\leq\big|(f\chi_{Q})(y)-(f\chi_Q)_{\mathcal{Q}_k}\big|+2^{n/p}\|f\|_{\mathrm{BMO}}.
\end{split}
\end{equation*}
Hence
\begin{equation*}
\begin{split}
\mathrm{IV}&\leq\frac{1}{|Q|^{\kappa}}\int_0^{\infty}\sum_{k=1}^\infty\big|\big\{y\in Q\cap \mathcal{Q}_k:|f(y)|>\lambda\big\}\big|q\lambda^{q-1}\,d\lambda\\
&\leq\frac{1}{|Q|^{\kappa}}\int_0^{\infty}\sum_{k=1}^\infty\big|\big\{y\in Q\cap \mathcal{Q}_k:\big|(f\chi_{Q})(y)-(f\chi_Q)_{\mathcal{Q}_k}\big|>\lambda-2^{\frac{n}{p}}\|f\|_{\mathrm{BMO}}\big\}\big|q\lambda^{q-1}\,d\lambda\\
&=\frac{1}{|Q|^{\kappa}}\bigg(\int_0^{2^{\frac{n}{p}}\|f\|_{\mathrm{BMO}}}+\int_{2^{\frac{n}{p}}\|f\|_{\mathrm{BMO}}}^\infty\bigg)\sum_{k=1}^\infty\\
&\big|\big\{y\in Q\cap \mathcal{Q}_k:\big|(f\chi_{Q})(y)-(f\chi_Q)_{\mathcal{Q}_k}\big|>\lambda-2^{\frac{n}{p}}\|f\|_{\mathrm{BMO}}\big\}\big|q\lambda^{q-1}\,d\lambda.
\end{split}
\end{equation*}
Observe that if $\lambda<2^{n/p}\|f\|_{\mathrm{BMO}}$, then for each cube $\mathcal{Q}_k$, it holds that
\begin{equation*}
\big\{y\in Q\cap\mathcal{Q}_k:\big|(f\chi_{Q})(y)-(f\chi_Q)_{\mathcal{Q}_k}\big|
>\lambda-2^{\frac{n}{p}}\|f\|_{\mathrm{BMO}}\big\}=Q\cap\mathcal{Q}_k,
\end{equation*}
which further implies that
\begin{equation}\label{222}
\begin{split}
&\sum_{k=1}^\infty\big|\big\{y\in Q\cap\mathcal{Q}_k:\big|(f\chi_{Q})(y)-(f\chi_Q)_{\mathcal{Q}_k}\big|
>\lambda-2^{\frac{n}{p}}\|f\|_{\mathrm{BMO}}\big\}\big|\\
&=\sum_{k=1}^\infty|Q\cap\mathcal{Q}_k|\leq\sum_{k=1}^\infty|\mathcal{Q}_k|\\
&<\frac{1}{\sigma}\sum_{k=1}^\infty\int_{\mathcal{Q}_k}\big|f(y)\cdot\chi_{Q}(y)\big|^p\,dy
=\frac{1}{\sigma}\int_{\bigcup_{k=1}^{\infty}\mathcal{Q}_k}\big|f(y)\cdot\chi_{Q}(y)\big|^p\,dy\\
&\leq\frac{1}{\sigma}\int_{Q}\big|f(y)\big|^p\,dy
\leq\left(\frac{\|f\|_{L^{p,\kappa}}}{\|f\|_{\mathrm{BMO}}}\right)^p|Q|^{\kappa}.
\end{split}
\end{equation}
Thus, we find that
\begin{equation*}
\begin{split}
&\frac{1}{|Q|^{\kappa}}\int_0^{2^{\frac{n}{p}}\|f\|_{\mathrm{BMO}}}\sum_{k=1}^\infty
\big|\big\{y\in Q\cap\mathcal{Q}_k:\big|(f\chi_{Q})(y)-(f\chi_Q)_{\mathcal{Q}_k}\big|
>\lambda-2^{\frac{n}{p}}\|f\|_{\mathrm{BMO}}\big\}\big|q\lambda^{q-1}\,d\lambda\\
&\leq\int_0^{2^{\frac{n}{p}}\|f\|_{\mathrm{BMO}}}
\left(\frac{\|f\|_{L^{p,\kappa}}}{\|f\|_{\mathrm{BMO}}}\right)^p q\lambda^{q-1}\,d\lambda\\
&=2^{\frac{nq}{p}}\big(\|f\|_{L^{p,\kappa}}\big)^p\cdot\big(\|f\|_{\mathrm{BMO}}\big)^{q-p}.
\end{split}
\end{equation*}
On the other hand, for fixed cube $Q$ and $f\in \mathrm{BMO}(\mathbb R^n)$, we first claim that $f\chi_Q\in \mathrm{BMO}(\mathbb R^n)$, and the following inequality holds:
\begin{equation}\label{diff}
\|f\chi_Q\|_{\mathrm{BMO}}\leq 2\|f\|_{\mathrm{BMO}}.
\end{equation}
In fact, for any cube $\mathcal{R}\subset\mathbb R^n$, by a simple calculation, we have
\begin{equation*}
\begin{split}
&\frac{1}{|\mathcal{R}|}\int_{\mathcal{R}}\big|(f\chi_{Q})(y)-(f\chi_{Q})_{\mathcal{R}}\big|\,dy\\
&\leq\frac{1}{|\mathcal{R}|}\int_{\mathcal{R}}\big|(f\chi_{Q})(y)-f_{\mathcal{R}}\big|\,dy
+\frac{1}{|\mathcal{R}|}\int_{\mathcal{R}}\big|f_{\mathcal{R}}-(f\chi_{Q})_{\mathcal{R}}\big|\,dy\\
&\leq \frac{2}{|\mathcal{R}|}\int_{\mathcal{R}}\big|(f\chi_{Q})(y)-f_{\mathcal{R}}\big|\,dy
=\frac{2}{|\mathcal{R}|}\int_{\mathcal{R}\cap Q}\big|f(y)-f_{\mathcal{R}}\big|\,dy\\
&\leq\frac{2}{|\mathcal{R}|}\int_{\mathcal{R}}\big|f(y)-f_{\mathcal{R}}\big|\,dy\leq 2\|f\|_{\mathrm{BMO}}.
\end{split}
\end{equation*}
This proves \eqref{diff} by taking the supremum over all the cubes $\mathcal{R}\subset\mathbb R^n$. Now, from Theorem \ref{JN} and inequality \eqref{diff}, it follows that
\begin{equation*}
\begin{split}
&\frac{1}{|Q|^{\kappa}}\sum_{k=1}^\infty\int_{2^{\frac{n}{p}}\|f\|_{\mathrm{BMO}}}^\infty
\big|\big\{y\in Q\cap \mathcal{Q}_k:\big|(f\chi_{Q})(y)-(f\chi_Q)_{\mathcal{Q}_k}\big|>\lambda-2^{\frac{n}{p}}\|f\|_{\mathrm{BMO}}\big\}\big|q\lambda^{q-1}\,d\lambda\\
&\leq\frac{1}{|Q|^{\kappa}}\sum_{k=1}^\infty\int_0^{\infty}
\big|\big\{y\in\mathcal{Q}_k:\big|(f\chi_{Q})(y)-(f\chi_Q)_{\mathcal{Q}_k}\big|>\nu\big\}\big|q\Big(\nu+2^{\frac{n}{p}}\|f\|_{\mathrm{BMO}}\Big)^{q-1}\,d\nu\\
&\leq\frac{eq}{|Q|^{\kappa}}\sum_{k=1}^\infty|\mathcal{Q}_k|\times
\int_0^{\infty}\exp\bigg\{-\frac{\nu}{2^n e\|f\chi_{Q}\|_{\mathrm{BMO}}}\bigg\}
\cdot\Big(\nu+2^{\frac{n}{p}}\|f\|_{\mathrm{BMO}}\Big)^{q-1}\,d\nu\\
&\leq\frac{eq}{|Q|^{\kappa}}\sum_{k=1}^\infty|\mathcal{Q}_k|\times
\int_0^{\infty}\exp\bigg\{-\frac{\nu}{2^{n+1}e\|f\|_{\mathrm{BMO}}}\bigg\}
\cdot\Big(\nu+2^{\frac{n}{p}}\|f\|_{\mathrm{BMO}}\Big)^{q-1}\,d\nu\\
&\leq eq\left(\frac{\|f\|_{L^{p,\kappa}}}{\|f\|_{\mathrm{BMO}}}\right)^p\int_0^{\infty}e^{-\mu}
\cdot\Big(2^{n+1} e\|f\|_{\mathrm{BMO}}\cdot\mu+2^{\frac{n}{p}}\|f\|_{\mathrm{BMO}}\Big)^{q-1}\big(2^{n+1}e\|f\|_{\mathrm{BMO}}\big)\,d\mu\\
&=eq\big(2^{n+1}e\big)^q\big(\|f\|_{L^{p,\kappa}}\big)^p\cdot\big(\|f\|_{\mathrm{BMO}}\big)^{q-p}
\int_0^{\infty}e^{-\mu}\left(\mu+\frac{1}{2^{\frac{n}{p'}+1}\cdot e}\right)^{q-1}d\mu,
\end{split}
\end{equation*}
where in the last inequality we have used the estimate \eqref{222}. As in the previous proof, one can also obtain
\begin{equation*}
\begin{split}
\int_0^{\infty}e^{-\mu}\left(\mu+\frac{1}{2^{\frac{n}{p'}+1}\cdot e}\right)^{q-1}d\mu
&=\int_{b}^{\infty}e^{-(u-b)}u^{q-1}du~\Big(b=\frac{1}{2^{\frac{n}{p'}+1}\cdot e}\Big)\\
&\leq e^b\Gamma(q).
\end{split}
\end{equation*}
Summarizing the estimates derived above, we conclude that
\begin{equation*}
\frac{1}{|Q|^{\kappa}}\int_{Q}|f(y)|^q\,dy\leq \mathcal{B}(p,q,n)\big(\|f\|_{L^{p,\kappa}}\big)^p\cdot\big(\|f\|_{\mathrm{BMO}}\big)^{q-p}.
\end{equation*}
Here
\begin{equation*}
\begin{split}
\mathcal{B}(p,q,n)&=\frac{q}{q-p}+2^{\frac{nq}{p}}+2^{(n+1)q}q\Gamma(q)e^{b+q+1}\\
&=\frac{q}{q-p}+2^{\frac{nq}{p}}+2^{(n+1)q}\Gamma(q+1)e^{b+q+1}\quad \mathrm{with}~~ b=\frac{1}{2^{\frac{n}{p'}+1}\cdot e}.
\end{split}
\end{equation*}
Thus, \eqref{wanghua12} holds. Taking $q$-th roots and the supremum over all cubes $Q$ in $\mathbb R^n$, we are done.
\end{proof}

\begin{rem}
$(1)$ As in Theorem \ref{thmc}, the constant $\widetilde{C}(p,q,n)$ in Theorem \ref{Thm3} also has the following form
\begin{equation*}
\widetilde{C}(p,q,n)=\mathcal{O}(q), \quad \mbox{as}~ q\to\infty,
\end{equation*}
that is, $C(p,q,n)\leq C_n\cdot q$ when $q$ is large with an absolute constant $C_n$ independent of $q$. Indeed, it is easy to see that
\begin{equation*}
\left(\frac{q}{q-p}\right)^{1/q}=\mathcal{O}(1),\quad \frac{b+q+1}{q}=\mathcal{O}(1)~(q\to\infty),
\quad \& \quad \big(2^{\frac{nq}{p}}\big)^{1/q}=2^{\frac{n}{p}}\leq 2^n.
\end{equation*}
Thus, the desired result follows from the well-known fact that $\Gamma(q+1)^{1/q}=\mathcal{O}(q)$ as $q\to\infty$.
Inspired by \cite[Theorem 2.2]{Kozono2}, it is natural to ask whether the growth order $q$ in Theorem \ref{Thm3} is sharp as $q\to\infty$(we believe this is true).

$(2)$ It is not clear whether the condition on $f$ in Theorem \ref{Thm3} can be weakened. We do not know whether the conclusion of Theorem \ref{Thm3} still holds provided $f\in WL^{p,\kappa}(\mathbb R^n)\cap\mathrm{BMO}(\mathbb R^n)$ for certain $p$ and $\kappa$.
\end{rem}

We proceed with some applications. Let $1\leq p<\infty$ and $0<\kappa<1$. For any given cube $Q$ in $\mathbb R^n$, by using the H\"{o}lder inequality, we get
\begin{equation*}
\begin{split}
\bigg(\int_{Q}|F(y)\cdot G(y)|^p\,dy\bigg)^{1/p}
&\leq\bigg(\int_{Q}|F(y)|^{2p}\,dy\bigg)^{1/{(2p)}}\bigg(\int_{Q}|G(y)|^{2p}\,dy\bigg)^{1/{(2p)}}\\
&\leq\Big(\|F\|_{L^{2p,\kappa}}|Q|^{\kappa/{(2p)}}\Big)\cdot\Big(\|G\|_{L^{2p,\kappa}}|Q|^{\kappa/{(2p)}}\Big)\\
&=\Big(\|F\|_{L^{2p,\kappa}}\cdot\|G\|_{L^{2p,\kappa}}\Big)|Q|^{\kappa/p}.
\end{split}
\end{equation*}
Then we have $F\cdot G$ is in $L^{p,\kappa}(\mathbb R^n)$ and
\begin{equation*}
\|F\cdot G\|_{L^{p,\kappa}}\leq\|F\|_{L^{2p,\kappa}}\cdot\|G\|_{L^{2p,\kappa}}.
\end{equation*}
Moreover, in view of Theorem \ref{Thm3}, we obtain that
\begin{equation*}
\begin{split}
\|F\cdot G\|_{L^{p,\kappa}}
&\leq\|F\|_{L^{2p,\kappa}}\|G\|_{L^{2p,\kappa}}\\
&\leq C(p,n)\Big(\|F\|_{L^{p,\kappa}}^{1/2}\|F\|_{\mathrm{BMO}}^{1/2}\cdot\|G\|_{L^{p,\kappa}}^{1/2}\|G\|_{\mathrm{BMO}}^{1/2}\Big)\\
&=C(p,n)\Big(\|F\|_{L^{p,\kappa}}\|G\|_{\mathrm{BMO}}\cdot\|G\|_{L^{p,\kappa}}\|F\|_{\mathrm{BMO}}\Big)^{1/2}
\end{split}
\end{equation*}
holds for all $F,G\in L^{p,\kappa}(\mathbb R^n)\cap\mathrm{BMO}(\mathbb R^n)$.
Finally, from the elementary inequality $2(ab)^{1/2}\leq a+b$, $a,b>0$, we have the following bilinear estimates in $L^{p,\kappa}(\mathbb R^n)\cap\mathrm{BMO}(\mathbb R^n)$.
\begin{equation*}
\|F\cdot G\|_{L^{p,\kappa}}\leq C(p,n)\Big(\|F\|_{L^{p,\kappa}}\|G\|_{\mathrm{BMO}}+\|G\|_{L^{p,\kappa}}\|F\|_{\mathrm{BMO}}\Big).
\end{equation*}
The Morrey space $L^{p,\kappa}(\mathbb R^n)$ was introduced (as a useful tool) to study the local behavior of solutions to second order elliptic partial differential equations. Later, this function space was found many important applications to some nonlinear elliptic and parabolic partial differential equations of the second order. We believe these bilinear estimates could be used in the study of certain problems(such as the global existence, uniqueness and regularity of weak solutions)in nonlinear elliptic and parabolic equations.

\end{document}